\documentclass[12pt,  reqno]{amsart}
\usepackage{amsmath, amssymb,amsthm}

\newtheorem{theorem}{Theorem}[section]
\newtheorem{lemma}{Lemma}[section]
\newtheorem{proposition}{Proposition}[section]

\newtheorem{corollary}{Corollary}[section]

\newtheorem{remark}{Remark}

\numberwithin{equation}{section}
\theoremstyle{definition}
\theoremstyle{remark}

\numberwithin{equation}{section}

\begin{document}

\title[Hsiung-Minkowski formulas and applications]{An extension of Hsiung-Minkowski formulas and some applications}

\author{Kwok-Kun Kwong}
\address{Department of Mathematics,  National Cheng Kung University, Tainan City 70101, Taiwan}
\email{kwong@math.ncku.edu.tw}

\renewcommand{\subjclassname}{\textup{2010} Mathematics Subject Classification}
\subjclass[2010]{Primary 53C40; Secondary 53C50}
\date{}

\maketitle

\begin{abstract}
We prove a generalization of Hsiung-Minkowski formulas for closed submanifolds in semi-Riemannian manifolds with constant curvature. As a corollary, we obtain volume and area upper bounds for $k$-convex hypersurfaces in terms of a weighted total $k$-th mean curvature of the hypersurface. We also obtain some Alexandrov-type results and some eigenvalue estimates for hypersurfaces.
\end{abstract}

\section{Introduction}\label{sec: intro}

  There has been a number of results in the literature on Minkowski-type formulas, which relate the integrals of different (weighted) $k$-th mean curvatures for closed oriented hypersurfaces or submanifolds in a Riemannian manifold (see e.g. \cite{alencar1998integral, chen1971integral, hsiung1956some, reilly1973variational, strubing1984integral}), under various assumptions. As a typical example, let us recall the classical Hsiung-Minkowski formulas \cite{hsiung1956some}: if $(M, g)$ is a space form and $\Sigma$ is a closed oriented hypersurface in $M$ with a unit normal vector field $\nu$, suppose $M$ possesses a conformal vector field $X$, i.e. the Lie derivative of $g$ satisfies $\mathcal L_Xg= 2 \alpha g$ for some function $\alpha$, then we have
  \begin{equation}\label{eq: typical}
    \int_\Sigma \alpha\sigma_k= \int_\Sigma \sigma_{k+1}\nu\cdot X.
  \end{equation}

Somewhat surprisingly, many geometric results can be deduced from these simple formulas, most notably rigidity results such as Alexandrov's theorem (\cite{montiel1991compact}) or various characterizations of certain hypersurfaces (e.g. \cite{alencar1998integral, koh1998characterization}).  It is interesting to know if the integrands in \eqref{eq: typical} can be less restrictive, and if so, to what extent can these formulas be applied to generalize the aforementioned results, which is the aim of this paper. Indeed, we give a simple generalization of Hsiung-Minkowski formulas for closed submanifolds in semi-Riemannian manifolds with constant curvature.
For example, as a special case of Theorem \ref{thm: main}, we have the following result:
\begin{theorem}\label{thm: 0}
Suppose $(M^n,g)$ has constant curvature and $\Sigma$ is a closed oriented hypersurface. Assume
$X\in \Gamma(\phi^*(TM))$ is a conformal vector field along $\Sigma$, and $f$ is a smooth function on $\Sigma$. Then for $0\le k\le n-2$,
  \begin{equation*}
    \int_\Sigma \alpha f \sigma_k =\int_\Sigma f\sigma_{k+1}\nu\cdot X- \frac{1}{(n-1-k){{n-1}\choose k}}\int_\Sigma \langle T_k(\nabla f), X^T\rangle .
  \end{equation*}
   Here $\sigma_k$ is the normalized $k$-th mean curvature, $\alpha$ is defined by $\mathcal{L}_Xg =\alpha g$,
      $\nu$ is a unit normal vector field and $X^T$ is the tangential component of $X$ onto $T\Sigma$.
\end{theorem}

The definition of $T_k$ will be given in Section \ref{sec: prelim}. We remark that the classical Hsiung-Minkowski formulas \cite{hsiung1956some} can be recovered by putting $f=1$ in the above formula. To the author's knowledge, these formulas are new, especially in the higher codimension case (cf. Theorem \ref{thm: main}), and generalize the integral formulas in \cite{hsiung1956some}, \cite{chen1971integral}, \cite{alencar1998integral}, \cite{alias2003integral} and \cite{strubing1984integral}. By choosing suitable $f$ in Theorem \ref{thm: 0}, we can obtain the following corollary:
\begin{corollary}\label{cor: 0.1}(Corollary \ref{cor: Euc area})
Suppose $\Sigma$ is a closed hypersurface embedded in $\mathbb{R}^n$ with $\sigma_{k}>0$ for some $1\le k\le n-1$. Then
  $$\mathrm{Area}(\Sigma)\le \int_\Sigma \sigma_1 r\le \cdots\le \int_\Sigma\sigma_{k} r^{k}$$
  and
  $$n\mathrm{Vol}(\Omega)\le \int_\Sigma r\le \int_\Sigma \sigma_1 r^2 \le \cdots \le \int_\Sigma \sigma_{k}r^{k+1}.$$
  Here $r=|X|$ and $\Omega$ is the region enclosed by $\Sigma$. The equality occurs if and only if $\Sigma$ is a sphere centered at $O$.
\end{corollary}
Similar inequalities hold in the hyperbolic space and the hemisphere as well. This generalizes the result in \cite{kwong1906new} and \cite{kwong2012}. (We remark that in \cite{kwong1906new}, a special case of Corollary \ref{cor: 0.1} is proved using the inverse mean curvature flow approach instead of using integral formulas. It is interesting to compare the two approaches.) As another corollary, we have the following extension of Alexandrov's theorem:
\begin{corollary}(A special case of Corollary \ref{cor: alex})
  Suppose $\Sigma$ is a closed hypersurface embedded in $\mathbb{R}^n$. Assume $f>0$, $f'\ge 0$ and there exists $1\le k \le n-1$ such that
    $\sigma_k f(r)$ is constant, where $r$ is the distance from $O$.
   Then $\Sigma$ is a sphere.
\end{corollary}
This and the other similar corollaries (Corollary \ref{cor: alex2}, \ref{cor: alex3}) generalize the Alexandrov-type results in \cite{ros1987compact}, \cite{montiel1991compact}, \cite{koh1998characterization}, \cite{koh2000sphere}, and \cite{aledo1999integral}. Finally we also give eigenvalue estimates (Theorem \ref{thm: lambda1}, \ref{thm: garay}, \ref{thm: stek}) for a class of elliptic operators which generalize the results in \cite{garay1989application}, \cite{grosjean2002upper}, and \cite{veeravalli2001first}.

The rest of this paper is organized as follows. In Section \ref{sec: prelim} we give the necessary definitions and preliminary results. In Section \ref{sec: main}, we prove the main results. A number of corollaries  are given in Section \ref{sec: cor}.

{\sc Acknowledgments}:
This work was conducted when the author was working as a Research Fellow at Monash University. He would like to thank Monash University for providing an excellent research environment.
\section{Preliminaries}\label{sec: prelim}

Let us fix the notations in this paper.
Let $\phi$ be an isometric immersion of an $m$-dimensional semi-Riemannian manifold $\Sigma$ into an $n$-dimensional semi-Riemannian manifold $(M,g)$.
We  use $\overline \nabla $ and $\nabla $ to denote the connection on $(M,g)$ and $\Sigma$ respectively.
The second fundamental form of $\Sigma $ in $M$ is defined by $A(X,Y)=-({\overline \nabla _XY})^\perp$ and is normal-valued.
We denote $A(e_i, e_j)$ by $A_{ij}$, where $\lbrace e_i\rbrace_{i=1}^m$ is a local orthonormal frame on $\Sigma$. For simplicity, we write $g(X, Y)$ as $X\cdot Y$ and the induced metric on $\Sigma$  is denoted by $\langle \cdot,\cdot\rangle$. For any normal vector field $\nu$ of $\Sigma$ in $M$, we define the scalar second fundamental form $A^\nu\in \mathrm{End}(TM)$ by
$\langle A^\nu(X), Y \rangle= A(X, Y)\cdot \nu$, and let $A^\nu(e_i)= \sum_{j=1}^m (A^\nu)_i ^j e_j$. If $\Sigma$ is a hypersurface, we choose $\nu$ to be the outward unit normal whenever this makes sense.

We define the $k$-th mean curvature as follows. If $k$ is even,
$$ H _k =  \frac 1{k!}\sum_{\substack{i_1,\cdots, i_k\\
j_1, \cdots, j_k}} \epsilon_{j_1\cdots j_k}^{i_1\cdots i_k}(A_{i_1j_1}\cdot A_{i_2j_2})\cdots (A_{i_{k-1}j_{k-1}}\cdot A_{i_kj_k}).$$

If $k$ is odd, the $k$-th mean curvature is a normal vector field defined by
$$ H _k =\frac 1{k!}\sum_{\substack{i_1,\cdots, i_k\\
j_1, \cdots, j_k}} \epsilon _{j_1\cdots j_k}^{i_1\cdots i_k} (A_{i_1j_1}\cdot A_{i_2j_2})\cdots  (A_{i_{k-2}j_{k-2}}\cdot A_{i_{k-1}j_{k-1}})A_{i_k j_k}.$$
We also define $H_0=1$. Here $\epsilon_{i_1 \cdots i_k}^{j_1\cdots j_k}$ is  zero if $i_k=i_l$ or $j_k=j_l$ for some $k\ne l$, or if $\lbrace i_1, \cdots, i_k\rbrace \ne \lbrace j_1, \cdots, j_k\rbrace$ as sets, otherwise it is defined as the sign of the permutation $(i_1, \cdots, i_k)\mapsto (j_1, \cdots, j_k)$. We also define the normalized $k$-th mean curvature as
$$\sigma_k =\frac{H_k}{{m\choose k}}.$$

In the codimension one case, i.e. $\Sigma$ is a hypersurface, by taking the inner product with a unit normal if necessary, we can assume $H_k$ is scalar valued. In this case the value of $H_k$ is given by
\begin{equation}\label{eq: codim 1}
H_k =\pm\sum_{i_1<\cdots< i_k}\lambda_{i_1}\cdots \lambda_{i_k}
\end{equation}
where $\{\lambda_i\}_{i=1}^n$ are the principal curvatures. This definition of $H_k$  is used whenever $\Sigma $ is a hypersurface.

Following \cite{grosjean2002upper} and \cite{reilly1973variational}, we define the (generalized) $k$-th Newton transformation $T_k$ of $A$ (as a $(1,1)$ tensor, possibly vector-valued) on $\Sigma$ as follows.\\
If $k$ is even,
$$ {(T_k)}_j^{\,i}= \frac 1 {k!}
\sum_{\substack{i_1,\cdots, i_k\\ j_1, \cdots, j_k}}
\epsilon^{i  i_1 \ldots  i_k}_{j  j_1 \ldots  j_k}
 (A_{i_1j_1}\cdot A_{i_2j_2})\cdots  (A_{i_{k-1}j_{k-1}}\cdot A_{i_kj_k}).
$$
If $k$ is odd,
$${(T_k)}_j^{\,i}= \frac 1 {k!}
\sum_{\substack{i_1,\cdots, i_k\\ j_1, \cdots, j_k}}
\epsilon^{i  i_1 \ldots  i_k}_{j  j_1 \ldots  j_k}
 (A_{i_1j_1}\cdot A_{i_2j_2})\cdots  (A_{i_{k-2}j_{k-2}}\cdot A_{i_{k-1}j_{k-1}})A_{i_kj_k}.
$$
We also define $T_0= I$, the identity map. Again, in the codimension one case, by taking the inner product with a unit normal if necessary, we can assume $T_k$ is an ordinary $(1,1)$ tensor and if $\lbrace e_i\rbrace_{i=1}^m$ are the eigenvectors of $A$, then
$$T_k(e_i)=\pm\sum_{\substack{i_1<\cdots <i_k\\
i\ne i_l}}\lambda_{i_1}\cdots \lambda_{i_k}e_i.$$

This definition of $T_k$ is used whenever $\Sigma$ is a hypersurface.
Alternatively, in the hypersurface case, $T_k$ can be defined recursively by (see e.g. \cite{reilly1973variational})
\begin{equation}\label{eq: T}
\begin{split}
  T_0={I}\quad\mathrm{ and }\quad& T_{k}=H_k {I} - AT_{k-1}\textrm{ for }k\ge 1.
\end{split}
\end{equation}
Here $A=A^\nu$, where $\nu$ is the unit normal to $\Sigma$.

We collect some basic properties of $T_k$ and $H_k$:
\begin{lemma}\label{lem: identities}
  We have
  \begin{enumerate}
    \item\label{item: 1}
    $\mathrm{tr}(T_k)= (m-k)H_k=(m-k){m\choose k}\sigma_k$.
    \item\label{item: 2}
    If $M$ has constant curvature, then
    $\mathrm{div}(T_k)=0$. i.e.
    $\sum_{i=1}^m\nabla _{e_i}(T_{k})^i_{j}=0.$ (Here $\nabla $ is the normal connection if $k$ is odd. ) If $k=1$ and $m=n-1$, we can assume $M$ is Einstein instead. If $k=0$, we can remove any assumption on $M$.
    \item\label{item: 5}
    If $k$ is even, then
    $\sum_{i,j=1}^m (T_k)_i^jA_{ij}=(k+1) H_{k+1}. $
    If $k$ is odd, then
    $\sum_{i,j=1}^m (T_k)_i^j\cdot A_{ij}=(k+1) H_{k+1}. $
  \end{enumerate}
\end{lemma}
\begin{proof}
These equations are well-known, at least in the codimension one case (e.g. \cite{marques1997stability} Lemma 2.1). They can be found e.g. in \cite{grosjean2002upper} Lemma 2.1, 2.2 and \cite{kwong2012inequality} Lemma 2.1. For \eqref{item: 2}, if $k=1$ and $m=n-1$, then by Codazzi equation, we have $\mathrm{div}(A)= dH_1$, which is equivalent to $\mathrm{div}(T_1)=0$ by \eqref{eq: T}. The assertions are trivial for $k=0$.
\end{proof}
\section{Main results}\label{sec: main}

In this section, we prove the main result (Theorem \ref{thm: main}) in this paper. It turns out that our result is an almost immediate consequence of a fairly simple divergence formula (Proposition \ref{prop: main}), which may have applications elsewhere. We use the notations in Section \ref{sec: prelim}.  Throughout this section, we also assume that $\Sigma$ is a closed and oriented semi-Riemannian manifold isometrically immersed in $(M,g)$. We  omit the area element $dS$ or volume element $dV$ in the integrals when there is no confusion.

\begin{proposition}\label{prop: main}
  Let $T$ be a symmetric $(1,1)$ tensor on $\Sigma$, $f$ be a smooth function on $\Sigma$ and $X$ be a vector field on a neighborhood of $\Sigma$ in $M$. Then
  $$\mathrm{div}(fT(X^T))=\langle T(\nabla f), X^T\rangle+ f(\mathrm{div}\;T)(X^T)+ \frac{1}{2}f\langle T^\flat, \phi^* (\mathcal{L}_Xg )\rangle
      -f\langle T^\flat,  A^{X^\perp }\rangle.$$
      Here $\mathrm{div}$ is the divergence on $\Sigma$, $T^\flat$ is the $(0,2)$-tensor defined by $T^\flat (Y, Z)= \langle T(Y), Z\rangle$, $X^T$ (resp. $X^\perp$) is the tangential (resp. perpendicular) component of $X$ and $\mathcal{L}_Xg $ is the Lie derivative of $g$.
\end{proposition}
\begin{proof}
  Let $Y$ be the vector field defined by $Y=f T(X^T)$. Locally, let $\{e_i\}_{i=1}^m$ be a local orthonormal frame on $\Sigma$ such that $\langle e_i , e_j\rangle= \mu_i \delta_{ij}$, $\mu=\pm 1$. Then  $Y= \sum_{j=1}^m Y^j e_j$, where
  $Y^j= f \sum_{i=1}^m \mu_i T_{i}^j X\cdot e_i$, and that $ \mu_i T_i^j = \mu_j T_j^i$. We can assume that $\nabla _{e_i}e_j(p)=0$ for all $i,j$. We compute the divergence of $Y$ at $p$:
  \begin{equation*}
    \begin{split}
      &\mathrm{div}(Y)\\
      =& \langle \nabla f, T(X^T)\rangle+ f(\mathrm{div}\;T)(X^T)+ f\sum_{i,j=1}^m \mu_iT_{i}^j (\overline \nabla _{e_j}X\cdot e_i+X\cdot\overline \nabla _{e_j}e_i)\\
      =& \langle T(\nabla f), X^T\rangle+ f(\mathrm{div}\;T)(X^T)+ \frac{1}{2}f\sum_{i,j=1}^m \mu_i T_{i}^j (\overline \nabla _{e_j}X\cdot e_i+ \overline \nabla _{e_i}X\cdot e_j)\\
      &-f\sum_{i,j=1}^m \mu_i T_i^j (X\cdot A_{ji})\\
      =& \langle T(\nabla f), X^T\rangle+ f(\mathrm{div}\;T)(X^T)+ \frac{1}{2}f\langle T^\flat, \phi^* (\mathcal{L}_Xg )\rangle
      -f\langle T^\flat, A^{X^\perp}\rangle.
    \end{split}
  \end{equation*}
\end{proof}

To proceed, let us recall that a vector field $X$ on $M$ is said to be a conformal (Killing) vector field if it satisfies
\begin{equation}\label{eq: conf}
  \mathcal{L}_X g = 2 \alpha g
\end{equation}
for some function $\alpha$ on $M$, and in this case, it is easy to see that $\alpha= \frac{1}{n}\overline {\mathrm{div}}(X)$. Here $\overline {\mathrm{div}}$ is the divergence on $M$. More generally, for an immersion $\phi$ of $\Sigma$ into $(M,g)$, a vector field $X\in \Gamma (\phi^*( TM)) $ is conformal along $\phi$ if $\overline \nabla _YX\cdot Z+ \overline \nabla _ZX\cdot Y =2\alpha \langle Y, Z\rangle$ for any tangential vector fields $Y, Z\in \Gamma(T\Sigma)$.

We now state and prove our first main result.
\begin{theorem}\label{thm: main}
Suppose $M$ has constant curvature.
 Assume
$X\in \Gamma(\phi^*(TM))$ is a conformal vector field along $\Sigma$ with $\alpha$ given by \eqref{eq: conf}, and $f$ is a smooth function on $\Sigma$.
\begin{enumerate}
  \item
  If $0\le k\le m-1$ is even, then
  \begin{equation}\label{eq: even}
    \int_\Sigma \alpha f \sigma_k =\int_\Sigma f\sigma_{k+1}\cdot X- \frac{1}{(m-k){m\choose k}}\int_\Sigma \langle T_k(\nabla f), X^T\rangle .
  \end{equation}
  \item
  If  $m=n-1$ (i.e. hypersurface), then
  \begin{equation}\label{eq: hypersurf}
    \int_\Sigma \alpha f \sigma_k =\int_\Sigma f\sigma_{k+1}\nu\cdot X- \frac{1}{(n-1-k){{n-1}\choose k}}\int_\Sigma \langle T_k(\nabla f), X^T\rangle .
  \end{equation}
   Here $\sigma_k$ and $ \sigma_{k+1}$ are scalars, $T_k$ is understood to be an ordinary $2$-tensor, and
      $\nu$ is a unit normal vector field.  If $k=1$, we can assume $M$ is Einstein instead. If $k=0$, we can remove any assumption on $M$.
\end{enumerate}
\end{theorem}

\begin{proof}
  Recall that $\sigma_k = \frac{H_k}{{m\choose k}}$. The result follows by applying Proposition \ref{prop: main} to $T=T_k$, and using Lemma \ref{lem: identities} and the divergence theorem.
\end{proof}
In general, it does not make sense to talk about $\int_\Sigma \sigma_k$ if $k$ is odd. Even when the normal bundle $NM$ is parallel so that $\int_\Sigma \sigma_k$ makes sense, our approach does not seem to produce a result similar to Theorem \ref{thm: main}, as we cannot produce the term $\sum_{i,j=1}^m(T_k)_i^j \cdot A_{ij}$ and apply Lemma \ref{lem: identities}.

Instead, we now take a different approach to derive a formula similar to \eqref{eq: even} for all (and in particular, odd) $k$, which is due to Str{\"u}bing \cite{strubing1984integral}.
Similar to Section \ref{sec: prelim}, for a family of normal vector fields $\nu_1, \nu_2, \cdots$ (not necessarily distinct), we define
$$H_k (\nu_1, \cdots, \nu_k)=   \frac 1{k!}\sum_{\substack{i_1,\cdots, i_k\\
j_1, \cdots, j_k}} \epsilon_{j_1\cdots j_k}^{i_1\cdots i_k}(A^{\nu_1})_{i_1}^{j_1}  \cdots(A^{\nu_k})_{i_k}^{j_k},$$
$$\sigma_k (\nu_1, \cdots, \nu_k)=\frac{H_k(\nu_1, \cdots, \nu_k)}{{m\choose k}},$$
and
$$ {(T_k(\nu_1, \cdots, \nu_k))}_j^{\,i}= \frac 1 {k!}
\sum_{\substack{i_1,\cdots, i_k\\ j_1, \cdots, j_k}}
\epsilon^{i  i_1 \ldots  i_k}_{j  j_1 \ldots  j_k}
 (A^{\nu_1})_{i_1}^{j_1}  \cdots  (A^{\nu_k})_{i_k}^{j_k}.
$$
Similar to Lemma \ref{lem: identities}, we have
\begin{lemma}\label{lem: id2}
  For $k\ge 1$, we have ($\mathrm{tr}$ denotes the trace on $\Sigma$):
  \begin{enumerate}
    \item\label{item: 1'}
    $\mathrm{tr}(T_k(\nu_1, \cdots, \nu_k))= (m-k)H_k(\nu_1, \cdots, \nu_k)$.
    \item\label{item: 2'}
    If $M$ has constant curvature, and $\nu_1,\cdots,\nu_k$ are parallel in the normal bundle, then
    $\mathrm{div}(T_k(\nu_1, \cdots, \nu_k))=0$. i.e.
    $\sum_{i=1}^m\nabla _{e_i}(T_{k}(\nu_1, \cdots, \nu_k))^i_{j}=0.$
    \item\label{item: 5'}
    $\sum_{i,j=1}^m (T_k(\nu_1, \cdots, \nu_k))_i^j (A^{\nu_{k+1}})_j^i=(k+1) H_{k+1}(\nu_1, \cdots, \nu_{k+1}). $
  \end{enumerate}
\end{lemma}
\begin{proof}
  The proof is exactly the same as in the codimension one case of Lemma \ref{lem: identities}, see e.g. \cite{marques1997stability} Lemma 2.1, except that in \eqref{item: 2'}, we need the fact that $\nabla _{e_i}(A^\nu)_j^k= \nabla _{e_j}(A^\nu)_i^k$ if $\nu$ is parallel.
\end{proof}
By applying Proposition \ref{prop: main} to $T_k(\nu_1, \cdots, \nu_k)$ and using Lemma \ref{lem: id2}, we obtain the following
\begin{theorem}\label{thm: main'}
Suppose $M$  has constant curvature.
Assume
$X\in \Gamma(\phi^*(TM))$ is a conformal vector field along $\Sigma$ with $\alpha$ given by \eqref{eq: conf}, $f$ is a smooth function on $\Sigma$ and $\nu_1, \cdots, \nu_k$ are (not necessarily distinct) normal fields to $\Sigma$ which are parallel in the normal bundle. Then
  \begin{equation}
  \begin{split}
    \int_\Sigma \alpha f \sigma_k(\nu_1, \cdots, \nu_k) =&
    \int_\Sigma f\sigma_{k+1}(\nu_1, \cdots, \nu_k, X^\perp)\\
    &- \frac{1}{(m-k){m\choose k}}\int_\Sigma \langle T_k(\nu_1,\cdots, \nu_k)(\nabla f), X^T\rangle .
  \end{split}
  \end{equation}
\end{theorem}
\begin{remark}
  If $m=n-1$,  then Theorem \ref{thm: main'} is reduced to \eqref{eq: hypersurf} in Theorem \ref{thm: main}.
\end{remark}

\section{Examples and applications}\label{sec: cor}
\subsection{Explicit formulas and inequalities in various spaces}
By substituting different $f$, $M$ and $X$ in Theorem \ref{thm: main}, we can obtain several corollaries.

First, a definition: we define $\mathbb R^{p,q}$ to be the vector space $\mathbb R^{p+q}$ equipped with the semi-Riemannian metric
  $dx_1^2+\cdots + dx_p^2-dx_{p+1}^2-\cdots - dx_{p+q}^2$.

\begin{corollary}\label{cor: R^n}
Suppose $\Sigma$ is a closed oriented $m$-dimensional semi-Riemannian manifold isometrically immersed in $\mathbb{R}^{p,q}$, where $p+q=n$, and $f$ is a smooth function on $\Sigma$. Let $X$ be the position vector.
\begin{enumerate}
  \item
  If $0\le k\le m-1$ is even, then
  \begin{equation*}
    \int_\Sigma   f \sigma_k =\int_\Sigma f\sigma_{k+1}\cdot X- \frac{1}{(m-k){m\choose k}}\int_\Sigma \langle T_k(\nabla f), X^T\rangle .
  \end{equation*}
  \item
  If  $m=n-1$ (i.e. hypersurface), then
  \begin{equation*}
    \int_\Sigma   f \sigma_k =\int_\Sigma f\sigma_{k+1}\nu\cdot X- \frac{1}{(n-1-k){{n-1}\choose k}}\int_\Sigma \langle T_k(\nabla f), X^T\rangle .
  \end{equation*}
   Here $\sigma_k$ and $ \sigma_{k+1}$ are scalars, $T_k$ is understood to be an ordinary $2$-tensor, and
      $\nu$ is a unit normal vector field.
      \item
       If there exists (not necessarily distinct) normal fields $\nu_1, \cdots, \nu_k$ to $\Sigma$ which are parallel in the normal bundle. Then
  \begin{equation*}
  \begin{split}
    &\int_\Sigma f \sigma_k(\nu_1, \cdots, \nu_k)\\
    =&\int_\Sigma f\sigma_{k+1}(\nu_1, \cdots, \nu_k, X^\perp)
    - \frac{1}{(m-k){m\choose k}}\int_\Sigma \langle T_k(\nu_1,\cdots, \nu_k)(\nabla f), X^T\rangle .
  \end{split}
  \end{equation*}
\end{enumerate}
\end{corollary}
\begin{proof}
  This is a direct consequence of Theorem \ref{thm: main}, Theorem \ref{thm: main'} and the fact that $\mathcal{L}_Xg(Y, Z)= \overline \nabla _Y X\cdot Z+\overline \nabla _Z X\cdot Y= Y\cdot Z+Z\cdot Y=2\langle Y,Z\rangle$.
\end{proof}
\begin{corollary}\label{cor: Rpq}
  Suppose $\Sigma$ is
  immersed in $\mathbb{R}^{p,q}$, where $p+q=n$. Then
  \begin{enumerate}
    \item
    For all odd $1\le k \le m$, we have
    $\int_\Sigma \sigma_k=0.$
    Here we regard $\sigma_k$ as a vector valued function.
    \item
    If $\Sigma$ is a hypersurface, then for all $0\le k\le m=n-1$, we have
    $\int_\Sigma \sigma_k \nu=0.$ Here we regard $\sigma_k$ as a scalar.
  \end{enumerate}

\end{corollary}
\begin{proof}
  This follows from Theorem \ref{thm: main} by putting $f=1$ and $X= E_i$, $i=1,\cdots, n$, where $\{E_i\}_{i=1}^n$ is the standard orthonormal basis of $\mathbb{R}^{p,q}$. As $E_i$ are Killing vector fields, we have $\alpha=0$ and the result follows. (Alternatively, this also follows from integrating the divergence of the vector field (more appropriately, an $n$-tuple of vector fields) $\sum_{i,j=1}^m(T_{k-1})^i_j X^je_i$ on $\Sigma$, where $X$ is the position vector and $X^j = \nabla _{e_j}X$ regarded as an $n$-tuple. For yet another proof, note that $\int_\Sigma \sigma_{k}\nu\cdot X=\int_\Sigma f \sigma_{k-1}$ is invariant under translation of $X$. )
\end{proof}

\begin{proposition}\label{prop: euc1}
  Suppose $\Sigma$ is
  a closed hypersurface immersed in $\mathbb{R}^n$. Let $X$ be the position vector, $r=|X|$, $0\le k\le n-2$ and $f$ be a smooth function on $\mathbb{R}$.
  Then we have
  $$\int_\Sigma f(r) \sigma_k=\int_\Sigma f(r)\sigma_{k+1}X\cdot \nu- \frac{1}{(m-k){m\choose k}}\int_\Sigma \frac{f'(r)}{r}\langle T_k(X^T), X^T\rangle$$
  and
  $$\int_\Sigma f(u) \sigma_k=\int_\Sigma u f(u)\sigma_{k+1}- \frac{1}{(m-k){m\choose k}}\int_\Sigma f'(u)\langle T_kA(X^T), X^T\rangle$$
  where $u=X\cdot \nu$.
\end{proposition}
\begin{proof}By an arbitrary small translation, we can assume $O\notin \Sigma$.
  The first equation follows from Theorem \ref{thm: main} and the observation that $r\nabla r=X^T$. The second equation follows by putting $f=f(u)$ noting that
  $\nabla (X\cdot \nu)= A(X^T)$.
\end{proof}
  The following corollary generalizes \cite{kwong2012} Theorem 3.2 (1) and also \cite{kwong1906new} Theorem 2.
\begin{corollary}\label{cor: Euc area}
  Suppose $\Sigma$ is a closed hypersurface immersed in $\mathbb{R}^n$ such that
 $\sigma_{k}>0$ for some $1\le k\le n-1$. Assume that
  $p\ge0$. Then we have
  $$ \int_\Sigma r^{p}\le \int_\Sigma \sigma_1 r^{p+1}\le \cdots \le \int_\Sigma \sigma_k r^{p+k}
  $$
  where $r=|X|$ and $X$ is the position vector. The equality occurs if and only if $\Sigma$ is a sphere centered at $O$. In particular, we have
  \begin{equation*}\label{eq: cor1}
    \mathrm{Area}(\Sigma)\le \int_\Sigma \sigma_1 r\le \cdots\le \int_\Sigma\sigma_{k} r^{k}
  \end{equation*}
  and
  \begin{equation}\label{eq: cor2}
    n\mathrm{Vol}(\Omega)\le \int_\Sigma r\le \int_\Sigma \sigma_1 r^2 \le \cdots \le \int_\Sigma \sigma_{k}r^{k+1}
  \end{equation}
  if $\Sigma$ is embedded.
  Here $\Omega$ is the region enclosed by $\Sigma$. The equality holds if and only if $\Sigma$ is a sphere centered at $O$.
\end{corollary}
\begin{proof}
  By \cite{marques1997stability} Proposition 3.2, if $\sigma_{k}>0$ on $\Sigma$, then $T_j$ is positive for $0\le j< k$. By applying Proposition \ref{prop: euc1} with $f=r^l$ and the Cauchy-Schwarz inequality, we can get the inequalities.  If the equality holds, then $X^T=0$ as $T_j>0$ for $0\le j< k$, but then $\nabla (|X|^2)=0$, which implies $\Sigma$ is a sphere centered at $O$. The converse is easy.
  The inequality \eqref{eq: cor2} follows from the fact that $n\mathrm{Vol}(\Omega)=\int_\Sigma X\cdot \nu\le \int_\Sigma r$.
\end{proof}

To state our next result, we first set up the notations.
Recall $\mathbb R^{p,q}=(\displaystyle \mathbb R^{p+q}, \sum_{i=1}^p dx_i^2-\sum_{i=p+1}^{p+q}dx_i^2)$. We use $\cdot$ to denote both the inner product on $\mathbb{R}^{p,q}$ and the semi-Riemannian metric on $M_{p,q}(\mu)$, as defined below.
  Let  $\mu=\pm1$ and $M_{p,q}(\mu)=\lbrace X\in \mathbb R^{p,q}: X\cdot X=\mu\rbrace$  be a pseudo-sphere in $\mathbb R^{p,q}$. It is easy to see that $M_{p,q}(\mu)$ is totally umbilic in $\mathbb R^{p,q}$ and in particular has constant curvature. Naturally, we can identify $T_X M_{p,q}(\mu)$ with a subspace in $\mathbb{R}^{p,q}$ and $V\in T_X M_{p,q}(\mu)$ if and only if $V\cdot X=0$.

  Let us recall that the classical Hsiung-Minkowski formulas \cite{hsiung1956some}: if $(M, g)$ is a space form and $\Sigma$ is a closed oriented hypersurface in $M$ with a unit normal vector field $\nu$. Suppose $M$ possesses a conformal vector field $Y$, i.e. the Lie derivative of $g$ satisfies $\mathcal L_Yg= 2 \alpha g$ for some function $\alpha$, then we have
  $$ \int_\Sigma \alpha\sigma_k= \int_\Sigma \sigma_{k+1}\nu\cdot Y. $$

  It is a nice observation that in general, if $(M, g)$ is a semi-Riemannian manifold which is isometrically embedded as a totally umbilic hypersurface in another semi-Riemannian manifold $(N, h)$, and such that there exists a conformal vector field $Z$ on $N$, then the orthogonal projection $Z^T$ of that vector field on $M$ is a conformal vector field on $M$. Indeed, a simple calculation shows that on $ TM$, if $\mathcal{L}_Zh= \alpha h$, then
  \begin{equation}\label{eq: L}
    \mathcal{L}_{Z^T}g= 2 \alpha g -2 A^{Z^\perp}.
  \end{equation}
  Therefore $Z^T$ is conformal on $M$ if $M$ is totally umbilic.
  In particular, we can construct a conformal vector field on  $M_{p,q}(\mu)$ by projecting any conformal vector field on $\mathbb R^{p,q}$ onto $M_{p,q}(\mu)$.

  In the following, we consider the special case where the conformal vector field $Y$ on $M_{p,q}(\mu)$ is the orthogonal projection of a constant vector field on $\mathbb R^{p,q}$. More precisely, fix $Z_0\in \mathbb R^{p,q}$, considered as a parallel vector field on $\mathbb R^{p,q}$. The orthogonal projection $Y$ of $-\mu Z_0$ (this choice will make the conformal factor looks neater) on $M_{p,q}(\mu)$ is then given by $-\mu Z_0=(- \mu Z_0)^T+ (-\mu Z_0)^\perp=  Y(X)-(Z_0\cdot X)X$, or equivalently,
  \begin{equation}\label{eq: Y}
    Y(X)=-\mu Z_0+(Z_0\cdot X)X\quad \textrm{ for $X\in M_{p,q}(\mu)$. }
  \end{equation}
  It is easily shown that the second fundamental form of $M_{p,q}(\mu) $ in $\mathbb{R}^{p,q}$ is
  $$ A(U,V)=\mu g(U,V)X\quad   \textrm{for $X\in M_{p,q}(\mu)$ and $U,V \in T_XM_{p,q}(\mu)$.}$$
  In particular, for $Y$ defined in \eqref{eq: Y}, in view of \eqref{eq: L}, we have
  \begin{equation}\label{eq: L_Yg}
    \mathcal{L}_Yg = 2 (Z_0 \cdot X )g \quad \textrm{at $X\in M_{p,q}(\mu)$.}
  \end{equation}
  By Theorem \ref{thm: main}, Theorem \ref{thm: main'}, and in view of \eqref{eq: L_Yg}, we have the following result:
\begin{theorem}\label{thm: 2}
 Let $\Sigma$ be an $m$-dimensional closed oriented semi-Riemannian manifold isometrically immersed in $M_{p,q}(\mu)$.
 Let
$f$ be a smooth function on $\Sigma$, $Z_0\in \mathbb{R}^{p,q}$ be fixed and $Y(X)$ be given by \eqref{eq: Y}.
\begin{enumerate}
  \item
  If $0\le k\le m-1$ is even, then
  \begin{equation*}
  \begin{split}
    \int_\Sigma (Z_0\cdot X) f \sigma_k dS(X)
    =&\int_\Sigma f\sigma_{k+1}\cdot Y(X) dS(X)\\
    &- \frac{1}{(m-k){m\choose k}}\int_\Sigma \langle T_k(\nabla f), Y^T\rangle dS(X).
  \end{split}
  \end{equation*}
  \item
  If  $m=n-1$ (i.e. hypersurface), then
  \begin{equation*}
  \begin{split}
    \int_\Sigma (Z_0\cdot X)f \sigma_k dS(X)
    =&\int_\Sigma f\sigma_{k+1}\nu\cdot Y(X)dS(X)\\
    &- \frac{1}{(n-1-k){{n-1}\choose k}}\int_\Sigma \langle T_k(\nabla f), Y^T\rangle dS(X).
  \end{split}
  \end{equation*}
   Here $\sigma_k$ and $ \sigma_{k+1}$ are scalars, $T_k$ is understood to be an ordinary $2$-tensor, and
      $\nu$ is a unit normal vector field of $\Sigma$ in $M_{p,q}(\mu)$.
      \item
       If there exists (not necessarily distinct) normal fields $\nu_1, \cdots, \nu_k$ to $\Sigma$ which are parallel in the normal bundle. Then
  \begin{equation*}
  \begin{split}
    &\int_\Sigma (Z_0 \cdot X) f \sigma_k(\nu_1, \cdots, \nu_k) dS(X)\\
    =&\int_\Sigma f\sigma_{k+1}(\nu_1, \cdots, \nu_k, Y^\perp)dS(X)
    - \frac{1}{(m-k){m\choose k}}\int_\Sigma \langle T_k(\nu_1,\cdots, \nu_k)(\nabla f), Y^T\rangle dS(X).
  \end{split}
  \end{equation*}
\end{enumerate}
\end{theorem}
  We can actually get rid of $Z_0$ in the formulas in Theorem \ref{thm: 2}. Indeed, by \eqref{eq: Y}, we have $\nu \cdot Y=-\mu \nu\cdot Z_0$ and $\langle T_k (\nabla  f), Y^T\rangle= -\mu T_k (\nabla f)\cdot Z_0$. Therefore, say, when $m=n-1$, the formula becomes
  $$\int_\Sigma f \sigma_k X\cdot Z_0+\mu \int_\Sigma f \sigma_{k+1}\nu\cdot Z_0-\frac{\mu}{(n-1-k){{n-1}\choose k}}\int_\Sigma T_k(\nabla f)\cdot Z_0=0.$$
  Since $Z_0$ is arbitrary, we have
  \begin{theorem}[Theorem \ref{thm: 2} restated]\label{thm: 2'}
 Let $\Sigma$ be an $(n-1)$-dimensional closed oriented semi-Riemannian manifold isometrically immersed in $M_{p,q}(\mu)$, where $p+q=n+1$.
 Let
$f$ be a smooth function on $\Sigma$, then
  as a vector in $\mathbb{R}^{p,q}$
  $$\int_\Sigma f \sigma_k X  +\mu \int_\Sigma f \sigma_{k+1}\nu  -\frac{\mu}{(n-1-k){{n-1}\choose k}}\int_\Sigma T_k(\nabla f)  =0.$$
  \end{theorem}

In the following, we apply Theorem \ref{thm: 2} to $M_{p,q}(\mu)$ for different $(p,q, \mu)$. For simplicity, we only give the result when $\Sigma$ is a hypersurface in $M_{p,q}(\mu)$ (and consequently, all $\sigma_k$ are scalars).

Let us consider the case where $(p,q, \mu)=(n+1, 0, 1)$ so that $M_{p, q}(\mu)=\mathbb S^n$. Choose $Z_0\in \mathbb S^n$ and $(r, \theta)$ be the geodesic polar coordinates around $Z_0$ on $\mathbb S^n$, where $\theta\in \mathbb S^{n-1}$. Then $$Y=\sin r\partial _r\quad \textrm{and}\quad Z_0\cdot X=\cos  r.$$
By Theorem \ref{thm: 2}, we have

\begin{corollary}\label{cor: sphere}
With the notations above,
let $\Sigma$ be a closed hypersurface in $\mathbb{S}^n$ and $\nu$ be its unit normal.
   Suppose $f$ is a smooth function on $\Sigma$. Then for $0\le k\le n-2$,
  \begin{equation*}
  \begin{split}
    \int_\Sigma \cos r f \sigma_k
    =&\int_\Sigma f\sigma_{k+1}\nu\cdot (\sin r \partial _r)
    - \frac{1}{(n-1-k){{n-1}\choose k}}\int_\Sigma \langle T_k(\nabla f), (\sin r \partial _r) ^T\rangle.
  \end{split}
  \end{equation*}
\end{corollary}
By substituting different functions $f$ in Corollary \ref{cor: sphere}, we have:
\begin{proposition}\label{prop: sphere2}
  With the same assumptions as in Corollary \ref{cor: sphere}, suppose $0\le k\le n-2$ and
   $\Sigma$ is contained in the open hemisphere centered at $Z_0$. Let $f$ be a smooth function on $\mathbb{R}$.
   Then we have
  $$\int_\Sigma f(r) \cos r\sigma_k=\int_\Sigma f(r) \sigma_{k+1}\nu\cdot Y- \frac{1}{(n-1-k){{n-1}\choose k}}\int_\Sigma \frac{f'(r)}{\sin r}\langle T_k(Y^T), Y^T\rangle$$
  and
  $$\int_\Sigma f(u)\cos r\sigma_k=\int_\Sigma uf(u)\sigma_{k+1}- \frac{1}{(n-1-k){{n-1}\choose k}}\int_\Sigma f'(u)\langle T_kA^\nu(Y^T), Y^T\rangle$$
  where $Y=\sin r \partial _r$ and $u=Y\cdot\nu$.
\end{proposition}
\begin{proof}By a slight perturbation, we can assume $Z_0\notin \Sigma$.
  The first equation follows from Corollary \ref{cor: sphere} and the fact that $\sin r\nabla r=Y^T$. The second equation follows by the fact that
  $\nabla (Y\cdot \nu)= A^\nu(Y^T)$.
\end{proof}
We have the following analogue of Corollary \ref{cor: Euc area}:
\begin{corollary}\label{cor: sphere3}
With the same assumptions as in Corollary \ref{cor: sphere}, suppose $\Sigma$ is contained in the open hemisphere centered at $Z_0$ and $\sigma_{k}>0$ for some $1\le k\le n-1$. Then
  \begin{equation}\label{ineq: sphere2}
  \mathrm{Area}(\Sigma)=\int_\Sigma \sigma_0\le \int_\Sigma \sigma_1\tan r\le \int_\Sigma \sigma_2 \tan^2 r\le \cdots \le \int_\Sigma \sigma_{k}\tan ^{k}r
  \end{equation}
  and
  \begin{equation}\label{ineq: sphere3}
   \int_\Sigma \sigma_0\cos r \le \int_\Sigma \sigma_1 \sin r \le \int_\Sigma \sigma_2 \tan r \sin r \le \cdots \le\int_\Sigma \sigma_k \tan^{k-1} r \sin r.
  \end{equation}
  If $\Sigma$ is embedded and $\Omega$ is the region in the hemisphere enclosed by $\Sigma$, then
  \begin{equation}\label{ineq: sphere vol}
    n\int_\Omega \cos r \le \int_\Sigma \sigma_0\tan r \cos r \le \int_\Sigma \sigma_1\tan^2 r \cos r \le \cdots \le\int_\Sigma \sigma_k\tan^{k+1} r \cos r.
  \end{equation}
 The equality occurs if and only if $\Sigma$ is a sphere centered at $Z_0$.
 \end{corollary}
\begin{proof}
  By \cite{marques1997stability} Proposition 3.2, if $\sigma_{k}>0$ on $\Sigma$, then $T_j$ and $\sigma_j$ are both positive for $0\le j< k$.
    Applying Proposition \ref{prop: sphere2}, we have
  \begin{equation*}
    \begin{split}
      &\int_\Sigma \sigma_j \tan ^j r
      =\int_\Sigma \sigma_j\frac{\tan ^j r}{\cos r}\cos r \\
      =& \int_\Sigma  \sigma_{j+1} \frac{\tan^{j} r}{\cos r} Y\cdot\nu-\frac{1}{(n-1-j){{n-1}\choose j}}\int_\Sigma \frac{\tan^j r}{\sin^2 r}(j \sec^2 r + \tan^2 r)\langle T_j( Y^T),Y^T\rangle\\
      \le&\int_\Sigma  \sigma_{j+1}\frac{\tan^{j} r}{\cos r} Y\cdot\nu\\
      \le&\int_\Sigma  \sigma_{j+1}\tan^{j+1} r.
    \end{split}
  \end{equation*}
  The inequality \eqref{ineq: sphere2} then follows by induction. The inequality \eqref{ineq: sphere3} is similar. For \eqref{ineq: sphere vol}, firstly we have
  $$n\int_\Omega \cos r=\int_\Omega \overline {\mathrm{div}}Y=\int_\Sigma Y\cdot \nu\le \int_\Sigma \sin r=\int_\Sigma \tan r\cos r \sigma_0.$$
  Similar to the above argument, we have, for $0\le j<k$,
  \begin{align*}
    \int_\Sigma \tan^{j+1} r \cos r\sigma_j=&\int_\Sigma \tan^{j+1}r \sigma_{j+1}\nu\cdot Y- \frac{j+1}{(n-1-j){{n-1}\choose j}}\int_\Sigma \frac{\tan ^{j-1}r}{\cos^3 r}\langle T_j(Y^T), Y^T\rangle\\
    \le&\int_\Sigma \tan^{j+2}r\cos r\sigma_{j+1}.
  \end{align*}
  The inequality \eqref{ineq: sphere vol} then follows by induction.

  If the equality case holds, then $\nabla r=0$ and so $\Sigma$ is a sphere centered at $O$. The converse is easy.
\end{proof}

For the case where $(p,q, \mu)=(n,1, -1)$ so that $M_{p, q}(\mu)=\mathbb H^n\sqcup \mathbb H^n$. We can choose $-Z_0\in \mathbb H^n$ and $(r, \theta)$ be the geodesic polar coordinates around $-Z_0$ on $\mathbb H^n$, where $\theta\in \mathbb S^{n-1}$. Then $$Y=\sinh r\partial _r\quad \textrm{and}\quad Z_0\cdot X=\cosh  r.$$
By Theorem \ref{thm: 2}, we have:
\begin{corollary}\label{cor: hyper}
With the notations above, let $\Sigma$ be a closed hypersurface in $\mathbb H^n$ with unit normal vector $\nu$.
   Suppose $f$ is a smooth function on $\Sigma$, then for $0\le k \le n-2$, we have
    $$\int_\Sigma f \cosh r \sigma_k= \int_\Sigma f\sigma_{k+1}\nu\cdot (\sinh r \partial _r)- \frac{1}{(n-1-k){{n-1}\choose k}}\int_\Sigma \langle T_k(\nabla f), (\sinh r \partial _r)^T\rangle.$$
\end{corollary}
\begin{proposition}\label{prop: hyper2}
  With the same assumptions as in Corollary \ref{cor: hyper}, suppose $0\le k\le n-2$ and
   $\Sigma$ is contained in the open hemisphere centered at $-Z_0$. Let $f$ be a smooth function on $\mathbb{R}$.
   Then we have
  $$\int_\Sigma f(r) \cosh r\sigma_k=\int_\Sigma f(r) \sigma_{k+1}\nu\cdot Y- \frac{1}{(n-1-k){{n-1}\choose k}}\int_\Sigma \frac{f'(r)}{\sinh r}\langle T_k(Y^T), Y^T\rangle$$
  and
  $$\int_\Sigma f(u)\cosh r\sigma_k=\int_\Sigma uf(u)\sigma_{k+1}- \frac{1}{(n-1-k){{n-1}\choose k}}\int_\Sigma f'(u)\langle T_kA^\nu(Y^T), Y^T\rangle$$
  where $Y=\sinh r \partial _r$ and $u=Y\cdot\nu$.
\end{proposition}
\begin{proof}By a slight perturbation, we can assume $-Z_0\notin \Sigma$.
  The first equation follows from Corollary \ref{cor: hyper} and the fact that $\sinh r\nabla r=Y^T$. The second equation follows by the fact that
  $\nabla (Y\cdot \nu)= A^\nu(Y^T)$.
\end{proof}
We have the following analogue of Corollary \ref{cor: Euc area}.
\begin{corollary}\label{cor: hyper3}
With the same assumptions as in Corollary \ref{cor: hyper}, suppose $\sigma_{k}>0$ for some $1\le k\le n-1$, then
\begin{equation}\label{ineq: hyper area}
  \int_\Sigma \sigma_0 \cosh r \le \int_\Sigma \sigma_1 \sinh r \le \int_\Sigma \sigma_2 \tanh r \sinh r \le\cdots \le\int_\Sigma \sigma_k \tanh^{k-1} r \sinh r.
\end{equation}
  Suppose $\Sigma$ is embedded and $\Omega$ is the region enclosed by $\Sigma$, then
  \begin{equation}\label{ineq: hyper vol}
    n\int_\Omega \cosh r \le \int_\Sigma \sigma_0\tanh r \cosh r \le \int_\Sigma \sigma_1\tanh^2 r \cosh r \le \cdots \le\int_\Sigma \sigma_k\tanh^{k+1} r \cosh r.
  \end{equation}
 The equality occurs if and only if $\Sigma$ is a geodesic sphere centered at $Z_0$.
 \end{corollary}
\begin{proof}
  By \cite{marques1997stability} Proposition 3.2, if $\sigma_{k}>0$ on $\Sigma$, then $T_j$ and $\sigma_j$ are both positive for $0\le j< k$. Let $Y=\sinh r \partial _r$.
Firstly we have
  $$n\int_\Omega \cosh r=\int_\Omega \overline {\mathrm{div}}Y=\int_\Sigma Y\cdot \nu\le \int_\Sigma \sinh r=\int_\Sigma \tanh r\cosh r \sigma_0.$$
  By applying Proposition \ref{prop: hyper2}, we have, for $0\le j<k$,
  \begin{align*}
    \int_\Sigma \tanh^{j+1} r \cosh r\sigma_j=&\int_\Sigma \tanh^{j+1}r \sigma_{j+1}\nu\cdot Y- \frac{j+1}{(n-1-j){{n-1}\choose j}}\int_\Sigma \frac{\tanh ^{j-1}r}{\cosh^3 r}\langle T_j(Y^T), Y^T\rangle\\
    \le&\int_\Sigma \tanh^{j+2}r\cosh r\sigma_{j+1}.
  \end{align*}
  The inequality \eqref{ineq: hyper vol} then follows by induction. The inequality \eqref{ineq: hyper area} is proved in a similar way as \eqref{ineq: sphere3}.

  If the equality case holds, then $\nabla r=0$ and so $\Sigma$ is a sphere centered at $O$. The converse is easy.
\end{proof}

For the case where $(p,q, \mu)=(n,1, 1)$ so that $M_{p, q}(\mu)=dS_n$, the de Sitter space. We choose $Z_0=(0, \cdots, 0, -1)$ and parametrize $dS_n$ by $X=(\cosh r \;\theta, \sinh r)$, where $\theta\in \mathbb S^{n-1}$.  Then
\begin{equation} \label{eq: dSn}
    Y=\cosh r\partial _r\quad \textrm{and}\quad Z_0\cdot X=\sinh  r.
\end{equation}
By Theorem \ref{thm: 2}, we have

\begin{corollary}\label{cor: dS}
With the notations above, let $\Sigma$ be an $(n-1)$-dimensional closed oriented semi-Riemannian manifold isometrically immersed in $d S_n$ with unit normal vector $\nu$.
   Suppose $f$ is a smooth function on $\Sigma$, then for $0\le k \le n-2$, we have
  $$\int_\Sigma f \sinh r \sigma_k= \int_\Sigma f\sigma_{k+1}\nu \cdot (\cosh r \partial _r)-  \frac{1}{(n-1-k){{n-1}\choose k}}\int_\Sigma \langle T_k(\nabla f), (\cosh r \partial _r)^T\rangle.$$
\end{corollary}
Similar to Proposition \ref{prop: sphere2}, we have
\begin{corollary}\label{cor: dS2}
  With the same assumptions as in Corollary \ref{cor: dS}, suppose $0\le k\le n-2$. Let $f$ be a smooth function on $\mathbb{R}$.
   Then we have
  $$\int_\Sigma f(r) \sinh r\sigma_k=\int_\Sigma f(r) \sigma_{k+1}\nu\cdot Y- \frac{1}{(n-1-k){{n-1}\choose k}}\int_\Sigma \frac{f'(r)}{\cosh r}\langle T_k(Y^T), Y^T\rangle$$
  and
  $$\int_\Sigma f(u)\sinh r\sigma_k=\int_\Sigma uf(u)\sigma_{k+1}- \frac{1}{(n-1-k){{n-1}\choose k}}\int_\Sigma f'(u)\langle T_kA^\nu(Y^T), Y^T\rangle$$
  where $Y=\cosh r \partial _r$ and $u=Y\cdot\nu$.
\end{corollary}

\subsection{Alexandrov-type results}
We have the following extension of Alexandrov's theorem.
\begin{corollary}\label{cor: alex}

Let $(M,g)$ be $\mathbb{R}^n$, $\mathbb{H}^n$, or $\mathbb{S}^n_+$ (the open hemisphere). Let $r$ be the distance on $M$ from a fixed point $O\in M$, taken to be the center if $M=\mathbb{S}^n_+$.
  Suppose $\Sigma$ is a closed  hypersurface embedded in $M$. Assume $f>0$, $f'\ge 0$ and there exists $1\le k \le n-1$ such that
  \begin{enumerate}
    \item
    $\sigma_k f(r)$ is constant,  or
    \item
    $\sigma_k f(u)$  is constant, where $u=X\cdot \nu$, and $\Sigma$ is convex (i.e. $A^\nu>0$). Here
    \begin{equation}\label{eq: X}
      X=\begin{cases}
   r\partial _r\quad &\textrm{if }M=\mathbb{R}^n\\
  \sin r \partial _r\quad &\textrm{if }M=\mathbb{S}^n_+\\
  \sinh r\partial _r\quad &\textrm{if }M=\mathbb{H}^n,
\end{cases}
    \end{equation}
  \end{enumerate}
  Then $\Sigma$ is a geodesic sphere, which is centered at $O$ if $f$ is injective.
\end{corollary}

\begin{proof}
  Assume first $\sigma_k f(r)$ is constant. Since $\Sigma$ has an elliptic point (i.e. point at which $A^\nu$ is definite,  cf. \cite{marques1997stability}), $\sigma_kf(r)$ must be positive and hence $\sigma_k>0$.
  By \cite{marques1997stability} Proposition 3.2,  then $T_j$ is positive definite and $\sigma_j>0$ for $0\le j<k$. Let
  \begin{equation}\label{eq: alpha}
    \alpha=\begin{cases}
  1\quad &\textrm{if }M=\mathbb{R}^n\\
  \cos r \quad &\textrm{if }M=\mathbb{S}^n_+\\
  \cosh r\quad &\textrm{if }M=\mathbb{H}^n,
\end{cases}
  \end{equation}
so by Proposition \ref{prop: euc1}, \ref{prop: sphere2} or \ref{prop: hyper2}, we have
  $$ \int_\Sigma\alpha \sigma_{k-1}f(r)
  \le \int_\Sigma \sigma_k f(r) X\cdot \nu
  =\sigma_k f(r) \int_\Sigma X\cdot \nu
  = \sigma_k f(r)n \int_  \Omega \alpha, $$
  where $\Omega$ is the region bounded by $\Sigma$.
  Therefore, by Newton's inequality,
  $$n\int_\Omega \alpha
  \ge \int_\Sigma\alpha\frac{\sigma_{k-1}}{\sigma_k}
   \ge\int_\Sigma \alpha\frac{\sigma_{k-2}}{\sigma_{k-1}} \ge\cdots
  \ge\int_\Sigma \alpha\frac{1}{\sigma_1}.$$
  On the other hand, by \cite{ros1987compact} Theorem 1 or  \cite{brendle2013constant} Theorem 3.5, we have
  $$\int_\Sigma \frac{\alpha}{\sigma_1}\ge n \int_\Omega \alpha,$$
  and the equality holds if and only if $\Sigma$ is  a geodesic sphere by \cite{ros1987compact} Theorem 1  or  \cite{brendle2013constant} Theorem 3.5 again.

   If $f$ is injective, then $\langle X, \nu\rangle$ is a positive constant as $\sigma_k$ is constant. In particular, $\Sigma$ is star-shaped w.r.t. $O$, and the furthest point $p_1$ and the nearest point $p_2$ from $O$ satisfy $\langle X, \nu\rangle (p_1)= \langle X, \nu\rangle(p_2)$. We conclude that $\Sigma$ is centered at $O$. The remaining case can be proved similarly.
\end{proof}

In the case where $\Sigma$ is immersed in a simply connected space form, we have the following generalization of the results in \cite{koh2001addendum, koh2000sphere}.
\begin{corollary}\label{cor: alex2}
Let $(M,g)$ be $\mathbb{R}^n$, $\mathbb{H}^n$, or $\mathbb{S}^n_+$. Let $r$ be the distance on $M$ from a fixed point $O\in M$, taken to be the center if $M=\mathbb{S}^n_+$.
  Suppose $\Sigma$ is a closed oriented hypersurface immersed in $M$. Assume $f>0$, $f'\ge 0$ and there exists $1\le l<k \le n-1$ such that
  $\frac{f(r)\sigma_{k} }{\sigma_{l}}$ is constant.
  Then $\Sigma$ is a geodesic sphere.
\end{corollary}
\begin{proof}
Let $\alpha$ and $X$ be defined by \eqref{eq: alpha} and \eqref{eq: X} respectively.
Since there exists an elliptic point on $\Sigma$,  by inverting the normal if necessary,  from \cite{marques1997stability} Proposition 3.2, $\sigma_{j+1}>0$ and $T_j>0$ for all $j< k$. Assume $l \ge 1$, then by Newton's inequality, we have
\begin{equation}\label{ineq: newton}
  0<a=\frac{f\sigma_{k}}{\sigma_{l}}\le \frac{f\sigma_{k-1}}{\sigma_{l-1}}.
\end{equation}
By \eqref{ineq: newton} and applying Proposition \ref{prop: euc1}, \ref{prop: sphere2} or \ref{prop: hyper2}, we have
  \begin{align*}
    \int_\Sigma \alpha f\sigma_{k-1}
    \le\int_\Sigma f\sigma_{k} X\cdot \nu
    =a\int_\Sigma \sigma_{l} X\cdot \nu
    =a\int_\Sigma \alpha \sigma_{l-1}
    \le\int_\Sigma \alpha f\sigma_{k-1}.
  \end{align*}
  We conclude that \eqref{ineq: newton} is an equality and so $\Sigma$ is totally umbilic. Therefore $\Sigma$ is a sphere (\cite{alencar1993first}).
\end{proof}

We also have the following partial extension of a result of Koh \cite{koh1998characterization}:
\begin{corollary}\label{cor: koh}
  Suppose $(M,g)$ is an $n$-dimensional Einstein manifold ($n\ge 3$) which possesses a conformal vector field $X$ with $\overline {\mathrm{div}}X>0$. Let $\Sigma$ be a closed oriented hypersurface immersed in $M$ with at least one elliptic point. Assume that $\frac{\sigma_2}{\sigma_1}$ is constant on $\Sigma$. Then $\Sigma$ is totally umbilic. Indeed $\sigma_1$ (and hence $\sigma_2$) is constant and $A^\nu= \sigma_1\langle \cdot,\cdot\rangle$ where $\nu$ is the unit normal of $\Sigma$.
\end{corollary}
\begin{proof}
  Since there exists an elliptic point on $\Sigma$, $\frac{\sigma_2}{\sigma_1}=a>0$. By a result of Garding \cite{garding1959inequality}, $\sigma_1>0$ and we have the inequality
  \begin{equation}\label{ineq: garding}
    0<a=\frac{\sigma_2}{\sigma_1}\le \frac{\sigma_1}{\sigma_0}.
  \end{equation}
  By Theorem \ref{thm: main}, we have
  \begin{align*}
    \int_\Sigma \alpha\sigma_1=\int_\Sigma \sigma_2 X\cdot \nu= a\int_\Sigma \sigma_1 X\cdot \nu= a \int_\Sigma \alpha \sigma_0\le \int_\Sigma \alpha \sigma_1.
  \end{align*}
  We conclude that \eqref{ineq: garding} becomes an equality, and thus $\Sigma$ is totally umbilic. We then have $A^\nu= \sigma_1 \langle \cdot, \cdot\rangle$, which is equivalent to $T_1-(n-2)\sigma_1 I=0$ by \eqref{eq: T}. By taking the divergence and using Lemma \ref{lem: identities}, we conclude that $\sigma_1$ is constant, and so is $\sigma_2$. (See also \cite{koiso1981hypersurfaces} Proposition 2.)
  \end{proof}

We have an analogue of Corollary \ref{cor: alex2} for $dS_n$, which partially generalizes \cite{aledo1999integral} Theorem 3:
\begin{corollary}\label{cor: alex3}
  Let $\Sigma$ be a closed oriented $(n-1)$-dimensional semi-Riemannian manifold isometrically immersed in $dS_n^+=dS_n \cap \{x_{n+1}>0\}$ and $r$ be defined as in \eqref{eq: dSn}. Assume $f>0$, $f'\ge 0$ and there exists $1\le l<k \le n-1$ such that
  $\frac{f(r)\sigma_{k} }{\sigma_{l}}$ is constant.
  Then $\Sigma$ is a totally umbilic round sphere. If $f$ is injective, then $\Sigma$ is a slice $\{r=\textrm{constant}\}$.
\end{corollary}

\begin{proof}
First of all, there exists an elliptic point on $\Sigma$. Indeed, take the point $p$ where $r$ is minimum, then by comparison principle, each of the principal curvatures is not more than $-\tanh r$, the principal curvature of the $r$-slice. Using Corollary \ref{cor: dS}, we can then proceed as in the proof of Corollary \ref{cor: alex2} to show that $\Sigma$ is totally umbilic. As $dS_n$ has constant curvature, the argument in the proof of Corollary \ref{cor: koh} then shows that $\sigma_1$ is a positive constant. By considering the point where $r$ is maximum, it is easy to see that $\Sigma$ is spacelike. So by the Gauss equation, as $dS_n$ has constant curvature $1$, the normalized scalar curvature of $\Sigma$ is less than $1$ (note $\nu\cdot \nu=-1$). By \cite{cheng1998spacelike}, we then conclude that $\Sigma$ is isometric to a sphere. Clearly if $f$ is injective, then $f(r)$ being constant implies $\Sigma=\{r=\textrm{constant}\}$.
\end{proof}

\subsection{Estimates for eigenvalues}

We now give a generalization of a result of Grosjean  on the upper bound of the Laplacian eigenvalue of a hypersurface, using our formulas.
\begin{theorem}(cf. \cite[Theorem 3.1]{grosjean2002upper})\label{thm: lambda1}
  Let $\Sigma$ be a closed $(n-1)$-dimensional Riemannian manifold immersed in a simply connected $n$-dimensional space form $M_K$ of curvature $K=0, \pm 1$. Suppose $\sigma_{k+2}>0$ for some $k\ge 0$. Assume in addition that the image of $\Sigma$ is contained in a geodesic ball of radius $\frac{\pi}{4}$ if $K=1$. Then
  $$\lambda_1(T_k)\le (m-k){m\choose k}\max_\Sigma\left(K \sigma_k+ \sigma_{k+2}\right) $$
  where $m=n-1$ and $\lambda_1(T_k)$ is the first eigenvalue the (positive if $\sigma_{k+2}>0$) second order differential operator $-\mathrm{div}(T_k\circ\nabla )=-\langle T_k, \nabla ^2\cdot\rangle$ on $\Sigma$. The equality holds if and only if $\Sigma$ is immersed as a geodesic sphere. (Note that $\lambda_1(T_0)$ is just the first Laplacian eigenvalue.)
\end{theorem}
\begin{proof}
  Let
    $s_K(r)
  =\begin{cases}
    r\; &\textrm{if }K=0\\
    \sin r\; &\textrm{if }K=1\\
    \sinh r\; &\textrm{if }K=-1
  \end{cases}$ and
  $c_K(r)=s_K'(r)
$.
  We use the following model for $M_K$:
$$M_K=\lbrace x\in \mathbb R^{n+1}:(x^0, x^1,\cdots, x^n)= (c_K(r), s_K(r) \theta ),\;  r\geq 0, \theta \in \mathbb S^{n-1}\rbrace$$
with metric induced from $\sum_{i=1}^n (dx^i)^2+ K(dx^0)^2.$
By applying a rigid motion of $M_K$, we can assume that
\begin{equation*}
\int_\Sigma (x^1,\cdots, x^n)=0.
\end{equation*}
If $K=1$, then by the assumption that $\Sigma$ is contained in some geodesic ball of radius $\frac{\pi}{4}$, we can ensure that $\Sigma$ is contained in the closed ball of radius $\frac{\pi}{2}$ centered at $(1,0\cdots, 0)$. Now, let $O=(1,0,\cdots,0)\in M_K$, $(r, \theta)$ be the geodesic polar coordinates around $O$, $Y=s_K(r)\partial _r$. By the min-max principle and Lemma \ref{lem: identities}, we have (see also \cite{veeravalli2001first})
  \begin{equation}\label{ineq: lambda}
    \begin{split}
       &\lambda_1(T_k) \int_\Sigma s_K^2 \\
       =&\lambda_1(T_k)\sum_{i=1}^n \int_\Sigma (x^i)^2\\
       \le&  \int_\Sigma \sum_{i=1}^n \langle T_k (\nabla x^i), \nabla x^i\rangle\\
       =&\int_\Sigma\sum_{j,l=1}^m \left(\sum_{i=1}^n(\nabla _{e_j}x^i)(\nabla _{e_l}x^i)+K(\nabla _{e_j}x^0)(\nabla _{e_l}x^0)-K(\nabla _{e_j}x^0)(\nabla _{e_l}x^0)\right)\left(T_k\right)_j^l\\
       =&\int_\Sigma \left(\mathrm{tr}\,T_k-K\langle T_k(\nabla c_K), \nabla c_K\rangle\right)\\
       =&\int_\Sigma \left((m-k){m\choose k}\sigma_k-K\langle  T_k(\nabla c_K), \nabla c_K\rangle\right).
    \end{split}
  \end{equation}
  We have (cf. \cite[Lemma 2.6]{heintze1988extrinsic}):
  $$K\int_\Sigma \langle T_k(\nabla c_K), \nabla c_K\rangle=(m-k){m\choose k}\int_\Sigma (\sigma_k c_K^2-c_K \sigma_{k+1} Y\cdot \nu).$$
  Indeed, this is a direct consequence of Corollary \ref{cor: R^n}, Proposition \ref{prop: sphere2} or Proposition \ref{prop: hyper2}, using the fact that $Y^T=-K\nabla c_K$ if $K\ne 0$.
 Plugging this into \eqref{ineq: lambda}, and using $1-c_K^2 = K s_K^2 $, we have
 \begin{equation}\label{ineq: lam2}
   \begin{split}
     \lambda_1 (T_k)\int_\Sigma s_K^2
     \le &(m-k ){m\choose k}\left(K \int_\Sigma \sigma_k s_K^2 +\int_\Sigma c_K \sigma_{k+1} Y\cdot \nu\right).
   \end{split}
 \end{equation}
 As in the proof of Corollary \ref{cor: alex}, we have $\sigma_{k+1}>0$ and $T_{k+1}>0$. Using Proposition \ref{prop: euc1}, \ref{prop: sphere2} or \ref{prop: hyper2}, we obtain  $$\int_\Sigma c_K \sigma_{k+1} Y\cdot \nu\le\int_\Sigma s_Kc_K \sigma_{k+1}\le \int_\Sigma s_K\sigma_{k+2}Y\cdot \nu\le \int_\Sigma s_K^2\sigma_{k+2}.$$
 Combining this with \eqref{ineq: lam2}, we conclude that
 $$\lambda_1(T_k)\le (m-k){m\choose k}\max_\Sigma(K \sigma_k+ \sigma_{k+2}). $$ If the equality holds, then it is easy to see from the above argument that $\nabla r=0$, and thus $\Sigma$ is an immersed sphere.
\end{proof}

\begin{remark}
  If $k=0$, then the assumption on $\sigma_2$ in  Theorem \ref{thm: lambda1} is equivalent to the scalar curvature $R$ of $\Sigma$ satisfies $R>(n-1)(n-2)K$, and the conclusion can be restated as $\lambda_1\le \frac{\max_\Sigma R}{n-2}.$
\end{remark}

By a slightly different argument, we have the following generalization of a result of Garay (\cite{garay1989application}):
\begin{theorem}\label{thm: garay}
  Suppose $\Sigma$ is a closed embedded hypersurface in $\mathbb{R}^n$ such that
 $\sigma_{k}>0$ for some $1\le k\le n-1$ and $\Omega$ is the region bounded by $\Sigma$, then
$$n\lambda(T_k)\mathrm{Vol}(\Omega)\le (m-k){m\choose k}\left(\max_\Sigma \sigma_1\right)\int_\Sigma \sigma_k$$
where $m=n-1$ and $\lambda_1(T_k)$ is the first eigenvalue the second order differential operator $-\mathrm{div}(T_k\circ\nabla )$ on $\Sigma$. The equality holds if and only if $\Sigma$ is a sphere.
\end{theorem}
\begin{proof}
We can assume that the center of mass is $0$.
By \eqref{ineq: lambda}, we have
$$\lambda_1(T_k) \int_\Sigma r^2 \le (m-k){m\choose k}\int_\Sigma \sigma_k.$$
On the other hand, by Corollary \ref{cor: Euc area}, we have
$n\mathrm{Vol}(\Omega)\le \max _\Sigma \sigma_1 \int_\Sigma r^2 .$
The result follows by combining these two inequalities. By Corollary \ref{cor: Euc area}, the equality holds if and only if $\Sigma$ is a sphere.
\end{proof}

To state our next result, we need to define the Steklov eigenvalues, as follows. Let $(M,g)$ be a compact Riemannian manifold with smooth boundary $\partial M=\Sigma$. The first nonzero Steklov eigenvalue is defined as the smallest $p\ne0$ of the following Steklov problem
\begin{equation}\label{eq: stek}
  \begin{cases}
  \overline \Delta f =0\quad &\textrm{on }M\\
  \frac{\partial f}{\partial \nu}=pf \quad &\textrm{on }\partial M
\end{cases}
\end{equation}
where $\nu$ is the unit outward normal of $\partial M$.
It is known that the Steklov boundary problem \eqref{eq: stek} has a discrete spectrum
$$0=p_0< p_1\le p_2\le\cdots \to \infty. $$
Moreover, $p_1$ has the following variational characterization (\cite[Theorem 11]{raulot2012first})
$$p_1=\min_{\int_{\partial M}f=0} \frac{\int_M|\nabla f|^2}{\int_{\partial M}f^2}.$$

We  now prove an upper bound of $p_1$ with the techniques similar to that in Theorem \ref{thm: lambda1}, and using some ideas of Grosjean \cite{grosjean2002upper}. Let $\Sigma$ be an $m$-dimensional Riemannian manifold isometrically immersed in $\mathbb{R}^n$. Again we assume $\{e_i\}_{i=1}^m$ be a local orthonormal frame on $\Sigma$. Suppose $T$ is a divergence free $(1,1)$-tensor on $\Sigma$, we then define a normal vector field $H_T$ by
$$H_T = \sum_{i,j=1}^m T_i^j A_{ij}.$$
Let $X$ be the position vector in $\mathbb{R}^n$. It is easy to see that (\cite[Lemma 2.3]{grosjean2002upper})
\begin{equation}\label{eq: div T}
  \frac{1}{2}\mathrm{div}(T \nabla (|X|^2))= \mathrm{tr}(T)-X\cdot H_T.
\end{equation}
For convenience, we also define on $\partial M$ $$\sigma_{-1}=X\cdot \nu.$$
The reason for this definition is that \eqref{eq: typical} holds (trivially) for $k=-1$ for hypersurface in $\mathbb{R}^n$ ($\alpha=1$).
The following result is the analogue of \cite[Theorem 2.3]{grosjean2002upper}:
\begin{theorem}\label{thm: stek}
  Suppose $M$ is a compact domain in $\mathbb{R}^n$ with smooth $(n-1)$-dimensional boundary $\partial M$ such that its $(k+2)$-mean curvature is positive for some $-1\le k\le m-2$, where $m=n-1$. Then we have
  $$p_1\int_{\partial M}\sigma_k\le n \left(\max _{\partial M}\sigma_{k+2}\right)\mathrm{Vol}(M)$$
  where $p_1$ is the first Steklov eigenvalue of $M$. The equality holds if and only if $M$ is a ball.
\end{theorem}

\begin{proof}
  Let us assume first $k\ge 0$, then by Proposition \ref{prop: euc1}, we have
  $$\int_{\partial M}\sigma_{k+1} |X|\le \int_{\partial M}\sigma_{k+2}|X|^2.$$
  So by \eqref{eq: div T} and Lemma \ref{lem: identities}, we have
  \begin{align*}
    (m-k)\int_{\partial M}H_k
    =\int_{\partial M}\mathrm{tr}(T_k)
    =&\int_{\partial M}X\cdot H_{T_k}
    =(k+1){m\choose {k+1}}\int_{\partial M}\sigma_{k+1} X\cdot \nu\\
    \le& (k+1){m\choose {k+1}}\int_{\partial M} \sigma_{k+1}|X|\\
    \le& (k+1){m\choose {k+1}}\int_{\partial M} \sigma_{k+2}|X|^2\\
    \le& (k+1){m\choose {k+1}}\left(\max_{\partial M} \sigma_{k+2}\right)\int_{\partial M} |X|^2.
  \end{align*}
  If $k=-1$, then
  $$\int_{\partial M}\sigma_{-1}\le\int_{\partial M}|X|\le \int_{\partial M}\sigma_1 |X|^2\le \max_{\partial M}\sigma_1\int_{\partial M}|X|^2.$$
  By a translation, we can assume that the center of mass is $0$, i.e. $\int_{\partial M}X=0$, and therefore
  $$p_1\int_{\partial M}|X|^2 =p_1\sum_{i=1}^n \int_{\partial M}(X^i)^2\le \int_M\sum_{i=1}^n |\nabla X^i|^2=n \mathrm{Vol}(M).$$
  Combining this with the above inequalities, we can get the result. It is easy to see from the above that the equality holds if and only if $M$ is a ball, noting that the Steklov eigenvalue of the unit ball in $\mathbb{R}^n$ is $1$.
\end{proof}

\begin{remark}
  Note that $\int_{\partial M}\sigma_{-1}=n\mathrm{Vol}(M)$. Therefore when $k=-1$, the estimate in Theorem \ref{thm: stek} becomes $$p_1\le \max_{\partial M}\sigma_1.$$
\end{remark}

\end{document}